\documentclass[a4paper,10pt]{article}
\usepackage[utf8]{inputenc}
\usepackage{amssymb}

\usepackage{amssymb,amscd}
\usepackage{amscd}
\usepackage{amsmath,latexsym,amssymb,amsthm}
\usepackage{mathtools}
\usepackage{txfonts}
\usepackage{tikz}
\usepackage{bbm}
\usepackage{hyperref}
\usepackage{cite}
\usepackage{float}
\usepackage{titlesec}
\usepackage{algorithm}
\usepackage{algorithmicx}
\usepackage[noend]{algpseudocode}
\usepackage{xcolor}
\usepackage{titlesec}
\usepackage{authblk}

\titleformat{\section}
{\normalfont\Large\bfseries}{\thesection}{1em}{}
\titleformat{\subsection}
{\normalfont\large\bfseries}{\thesubsection}{1em}{}
\titleformat{\subsubsection}
{\normalfont\normalsize\bfseries}{\thesubsubsection}{1em}{}
\titleformat{\paragraph}[runin]
{\normalfont\normalsize\bfseries}{\theparagraph}{1em}{}
\titleformat{\subparagraph}[runin]
{\normalfont\normalsize\bfseries}{\thesubparagraph}{1em}{}

\theoremstyle{plain}
\newtheorem{thm}[subsection]{Theorem}
\newtheorem{lem}[subsection]{Lemma}
\newtheorem{prop}[subsection]{Proposition}

\newtheorem{conj}[subsection]{Conjecture}

\DeclarePairedDelimiter\floor{\lfloor}{\rfloor}
\theoremstyle{definition}

\newtheorem{rem}[subsection]{Remark}
\newtheorem{defn}[subsection]{Definition}

\title{Approximation of Banzhaf index and its application to voting games}
\author{Acharya V. Krishna  \thanks{krishnacharya97@gmail.com}}
 \author{Himadri Himadri  \thanks{himadrim@goa.bits-pilani.ac.in}}

  \author{Jajati Keshari Sahoo  \thanks{jksahoo@goa.bits-pilani.ac.in}}

\date{\today}
\affil[1]{Department of Mathematics, BITS Pilani K K Birla Goa Campus}
\begin{document}

\maketitle
\setlength{\parindent}{0pt}
\begin{abstract} In this paper, we propose an improved version of the power index related to the Banzhaf power index for weighted voting systems.
This index now takes into account the mutual persuasion power matrix(PPM) existing among the voters. This improved index is calculated for European
Union voting by basing the PPM on immigration data among the EU countries. We also provide better approximation bounds for the Monte Carlo approximation
method for computing power indices. \\
\noindent \textbf{keywords:} Voting Games, Power Indices, Banzhaf Index, Monte Carlo Method, EU Immigration\\
\noindent \textbf{2010 MSC classification:} 91A80. 62P20. 97M40. 91B82. 91A40
\end{abstract}

\section{Introduction}
In the year 1963 the supreme court of the United States gave the historical order of one person one vote, several scientists justified the decision using mathematical
analysis of the weighted voting system. The power index for weighted voting systems developed by Banzhaf in 1965, in the article (Banzhaf 1965), is one of the pioneers in this field. A good introduction of the theory is the original paper by Banzhaf \cite{banzhaf}, a relatively modern approach can be found in \cite{dubey, dubey, S_values}.
Probabilistic models that construct the index as probabilities can be found in \cite{SS_prob, prob_model}. There is also an axiomatic approach that constructs the indices through an axiomatization that
shows the uniqueness of such indices when a set of axioms are satisfied. There are numerous approaches to streamline the axiomatization as in \cite{dubey, axiom2, axiom3}.
Let $p_1,p_2,p_3, \ldots , p_m$ be a number of players with weights $w_1, w_2, w_3, \ldots, w_m$. A \textit{coalition} C, subset of the set of players, is called a wining coalition with respect to a preassigned quota $q$ if $\displaystyle{\sum_{p_i \in C} w_i \geqslant q}$.
Let the set of wining coalitions be $W_q$ or simply $W$ if the quota is
clear form the context.
\begin{defn}  A player $p_i$ is called critical or swinger for a coalition $C$ if the following two conditions are met:
\[ C \in W \]
\[ C -\{p_i\} \notin W \]
\end{defn}
Let us call $(p_i, C)$ a critical pair when $p_i \in C$ is critical.
Let $\mathfrak{C}_i$ be the set of all pairs where $p_i$ is critical. The Banzhaf power index for the player $p_i$ is denoted by $\beta_i$ (or $\tilde{\beta_i}$).
\[ \mathfrak{C}_i = \{ (p_i, C)| p_i \in C \mbox { is critical} \}\]
\[\gamma_i= \left\vert{\mathfrak{C}_i}\right\vert\]
\[ {\beta_i}= \gamma_i/ 2^{m-1} \]
\[ \tilde{\beta_i} = \gamma_i / \sum_{j=1}^m \gamma_j\]

$\tilde{\beta_i}$, is known as the normalized Banzhaf index.
$\beta_i$, the absolute Banzhaf index will, however, be our main focus for this article.

The index, though coming to life as a refutal of weighted voting systems, sees a lot of applications in real life. The applications are numerous in social science
\cite{naive, compound}.
In Electoral College voting - members of the parliament possess weighted votes in proportion to the number of constituencies they represent. In intergovernmental organization of nations such as the EU the member countries get their voting weights according to a number of parameters such as population etc \cite{Eu_fair, EU_tempt}. In the context of EU it is a great debate among the researchers if the voting power game used to determine the voting weights is fair \cite{Eu_fair} and often contains warning against such use \cite{EU_tempt}.
In this paper, we will try to explain yet again why the voting power method used in the EU Parliament may not be the right approach to find the voting rights, now from the perspective of the association between the players. It is evident that only the absolute quantity of the weights measured in economic terms, population, area fail to capture the true `power of a member' since such power should also include the diplomatic prowess of a nation. In the industry, search engines such as Google
use page rank algorithms applying techniques of weighted voting games among others. From a purely theoretical point of view the
index has generated interests in the line of inverse problems \cite{inverse, weber} which is to construct the weights of the players from the indices, or relate the indices with convex geometry \cite{Convex} or mathematical extensions \cite{extension}.
The Banzhaf index has been a source of great activities recently.
The recent study includes an attack on the computational challenges on the NP complete\cite{NP} problem. There have been multiple algorithms that tackle the computational difficulties of the index\cite{Approximations, genetic} or find theoretical bounds \cite{Bounds}. A good place to look at different approaches to approximating the index is the survey \cite{survey}. The Monte Carlo method by the \cite{Monte1, Monte2} uses Hoeffding's inequality to give a very
strong approximation within a large confidence interval. In \cite{Heu} the authors consider heuristics based on density functions. In \cite{Large} the authors employ generating functions.

In this paper, we introduce an improved version of the power index that takes into account situations where not only the individual weights of the players are considered but the associations between players as well. To explain the situation let us consider a set of politicians discussing a motion, irrespective of their partisan interests they can vote or abstain according to their personal preferences and equations. One obvious example is the Indian upper house (Rajya Sabha) debates where the members are frequently persuaded to vote based on issues related to the constituency they represent not completely complying with the partisan goals. Thus we can often form a matrix of associations between the players that does not necessarily represent the ideology based political decisions. In other fields such as geopolitics, we often see the actual strength of a nation is not merely the GDP or the might of its military or population but also the amount of influence that it can exert on the other nations by several means of \textit{soft power}. In this paper, we quantify the soft
power of a nation based on the association matrix and the improved power index developed in this model. To be precise let us say the players $\{p_1,p_2,p_3, \ldots, p_m\}$ are
engaged in a voting game. And $a_{ij}\in \mathbb{R}$ are numbers that quantify \textit{strength} of the association between the players, this represents the persuasion power of the
player $p_i$ has over $p_j$. We postulate that the strength of a coalition is not only the sum of the weights of the players but also the amount of total persuasive power that it
holds to sway the motion in its favour by converting other players. Thus the definition of a swinger is changed so that the diplomatic loss of a player leaving the coalition is
taken into account. In the real life, the matrix $A=(a_{ij})$ will almost surely be not symmetric, as the 'associations' are not symmetric. Ideas to restrict coalitions or equivalently
the critical pairs can be observed in \cite{Convex} where the authors restrict the \textit{allowed} coalitions to those coalitions forming a convex geometry, this is claimed to represent the EU scenario better. In our scenario, we will restrict the critical pairs in a uniform way to get control over the critical pairs.

We see a lot of applications in the fields of social sciences and computer science arising out of the improved Banzhaf power index, we list a few such directions at the end of this
article.

\section{Mathematical Model and Methodology}
In the recent studies of power indices in voting games, We see a systematic study to improve the applicability of the index to the practice of coalition formation (Burgin, 2001) as well as studies to improve the computational aspects of the power indices \cite{banzhaf, banzhaf68, straf, survey}. It has not come to our notice any attempt to propose a model where a voter’s individual relation (association) with another voter influences the outcome of the coalition formation. For an example consider a game with six players $a_1,a_2,\cdots, a_6$ wherein $a_1$ has very strong positive association with the other players, then any coalition containing $a_1$ will have negative impact on the other coalition effectively reducing the total weight of the opponent coalition and increasing the weight of coalition containing $a_1$. Such scenarios are clearly visible in Parliamentary and legislative assembly election, for instance, members of parliament have strong positive associations with other members having similar political motives. An attempt using the classical power indices or current variations of them will not properly predict the actual strength of a coalition in such realistic scenarios.
\subsection{Mathematical Formulation}
Let $\{p_1,p_2,\cdots, p_m\}$ be a set of players, with weights $w_1,w_2,\cdots,w_m$ respectively. Consider association between $p_i$ and $p_j$ as $a_{ij}\in\mathbb{R}$, with $|a_{ij}|\leq 1$ and $a_{ii}=1$ for all $i.$ The matrix $\Phi=(a_{ij})$ is called the association matrix for the players and the goal is to propose a power index $\beta_\Phi(p_i)$ which takes the association matrix into account. It is interesting to note that if all off-diagonal elements in $\Phi$ are zero i.e $a_{ij}=0, ~ \forall ~ i\neq j$, then calculations for this formulation reduces to a $k-$quota banzhaf index as studied in \cite{Heu}
\subsubsection{Single quota voting game}\label{ssec:mathdef}
Consider a voting game with association matrix $\Phi$ and a quota $q>0$.
Let S be the set of all possible coalitions.
$C\subseteq S$ is  called a winning coalition if
\begin{equation}\label{eq:1}
  \sum_{p_i\in C}w_i\geq q.
\end{equation}
Let $W_q=\{C\subseteq S~:~ C~ \mbox{ is wining}\}$ be the set of winning coalition. A tuple $(p_i, C)$ is called a critical pair if

\begin{equation}\label{eq:2}
 \mbox{(i) } p_i\in C \mbox{ and } C\in W_q~~~ \mbox{ (ii)}\sum_{p_j\in C}w_j-\sum_{k=1}^{m}a_{ik}w_k< q.
\end{equation}

The improved Banzhaf power index for the player $p_i$ is defined as:
\begin{equation}\label{eq:3}
  \beta_{\Phi,i}=\frac{\vert\gamma_i\vert}{2^{m-1}}
\end{equation}

where ~ $\gamma_i = \{(p_i,C)~ \vert ~(p_i,C) ~ \mbox{is critical}\}$

One can extend the definition to multi-weight and multi-quota voting systems in the following
manner. We call a vector $v=(v_1,v_2,\cdots,v_k)$ non negative if $v_i\geq 0$ for all $i$. For any
$u,v\in\mathbb{R}^k$, $u\geq v$ if $u-v$ is non-negative. In $k-$dimensional voting game, the weight for player $p_i$ and the quota required for winning coalitions are defined as:

\begin{equation}\label{eq:5}
  \vec{w_i}=(w_1,w_2,\cdots, w_k),~ \vec{q}=(q_1,q_2,\cdots, q_k)
\end{equation}

If we follow equation \ref{eq:2} and vector comparison we can get critical pairs for \\ $k-$dimensional voting games. The improved Banzhaf index for multidimensional voting game is again defined by equation \ref{eq:3}. Multidimensional voting games are used in many real life problems such as the EU election \cite{algaba}.

\subsection{Calculation of the Banzhaf index and its Approximation}\label{ssec:Calc}
\begin{algorithm}[H]
	\caption{$\beta_{\Phi,i}$ for a k-quota game with an association matrix $\Phi$ of dimension m$\times$ m}\label{algo:1}
	\begin{algorithmic}
		\Procedure{Exact-Banzhaf-Asso}{$\Phi,\vec{w}, \vec{q}$}
		\State $\gamma_i$ = 0
		\State \textbf{for} \textit{all} coalitions C containing $p_i$ \textbf{do}
		\State ~~~~\textbf{if} $(p_i,C)$ is a critical pair as defined in the extension \ref{eq:5}
		\State  ~~~~~~~~then $\gamma_i = \gamma_i+1$
		\State ~~~~\textbf{end if}
		\State {\textbf{end for}}
		\State $\beta_{\Phi,i}=\frac{\gamma_i}{2^{m-1}}$
		\State \hspace*{-0.6cm}{\textbf{return $\beta_{\Phi,i}$}}
		\EndProcedure
	\end{algorithmic}
\end{algorithm}

Exact calculation of the Banzhaf index for a general case is NP-complete as shown in \cite{NP}, so we need to find an alternative to Algorithm \ref{algo:1}.

In \cite{Monte1}, the authors have used	a Monte Carlo method to approximate the power indices for the players. We extend this idea for multi-quota games with associations. The approximate banzhaf index for player $p_i$, ~$\hat\beta_{\Phi,i}$  will be estimated by sampling random coalitions containing player $p_i$, then calculating the proportion of these samples where $p_i$ is a critical. We assume that each sample has a probability $\beta_{\Phi,i}$ of being a coalition where player $p_i$ is critical, so we can approximate $\beta_{\Phi,i}$  by taking into consideration several such samples. More precisely the estimation procedure is as follows:\\
Let $C_1,C_2,\cdots,C_n$ be $n$ randomly sampled coalitions containing player $p_i.$ For $1\leq k\leq n,$ define a Bernoulli random variable
$$X_k=\left\{\begin{array}{cc}
  1& ~~\mbox{ If player $p_i$ is critical in $C_k.$}\\
  0 &~~\mbox{otherwise.}
\end{array}\right.$$
Now, we estimate the index $\beta_{\Phi,i}$ for $p_i$ by using the following estimator
\begin{equation}\label{eq4}
  \hat{\beta}_{\Phi,i}=\frac{\sum_{k=1}^{n}X_k}{n}
\end{equation}
Algorithm \ref{algo:2} is used to compute the improved Banzhaf indices. The number of samples required, $n$, error bounds and convergence analysis is discussed in Section 4.
\begin{algorithm}[H]
\caption{$\hat{\beta_{\Phi,i}}$ for a $k-$quota game with an association matrix $\Phi$}\label{algo:2}
\begin{algorithmic}
\Procedure{Approx-Banzhaf-Asso}{$\Phi,\vec{w}, \vec{q}, n$}
\State $\gamma_i$ = 0, j = 0
\State {\textbf{while}}{ $j \le n$ \textbf{do}}
\State ~~~~Choose a coalition C at \textbf{random} containing $p_i$
\State ~~~~{\textbf{if}}{~$(p_i,C)$ is a critical pair as defined in equation \ref{eq:2}}
\State ~~~~~~~~then $\gamma_i$ =$\gamma_i$ + 1
\State ~~~~{\textbf{end-if}}
\State j = j + 1
\State {\textbf{end while}}
\State $\hat{\beta_{\Phi,i}}$ = $\frac{\gamma_i}{j}$
\State \hspace*{-0.6cm}{\textbf{return $\hat\beta_{\Phi,i}$}}
\EndProcedure
\end{algorithmic}
\end{algorithm}
\section{Some theoretical results}
Given a player $p_i$ let us define two functions $h(p_i)$ or simply $h(i)= \min \{ h \, \vert \, w_1+w_2+\cdots + w_h > q , \, \forall i , \, w_i \neq w\}$
and $t(i)= \max \{ t \, \vert \, w_1+w_2 + w_3 + \cdots + w_t < q \, \exists i , \, w_i =w  \}$. Note that in the definitions of $h(i),t(i)$ we will consider the
weights of distinct players as distinct $w_i$s even if they are same! With this notation in place we have the following estimates for the Banzhaf
index of a player $p_i$.

\begin{prop} $\beta_i \leq \, \frac{1}{2^{n}} (2^{n} - 2^{t(i)} - 2^{n-h(i)})$

\end{prop}
\begin{proof}
The set of coalitions that are subsets of the set of players $w_1,w_2, \ldots , w_{t(w)}$ such that $w_1+w_2 + w_3 + \cdots + w_t < q \, \exists \; i , \, w_i =w$
is not at all winning coalitions so any such coalitions containing the player
$i$ is not a critical coalition for the player. Similarly in the other side any superset containing the player $i$, of the set of players with weights $w_1,w_2, \ldots, w_{h(i)}$ such that
$w_1+w_2+\cdots + w_h > q , \, \forall i , \, w_i \neq w$ is not a critical coalition for the player as
the coalition is winning without the player $i$. Now $t$ being maximum and $h$ being the minimum we have the result.
\end{proof}
In the above proposition the greater the variation in the weights of the players better the estimate from the above inequality. In particular if there is a dominating weight
or a very large quota or equivalently a very tiny quota the estimate provided by the above becomes close. Other than that it is of theoretical interest as the estimate provided by the above is not practical. It will be interesting to study in detail such estimates.
\begin{lem}\label{bound3}
 Let $C$ be a coalition, $w(C)$ be the total weight of $C$ and $q$ the quota. If every player in $C$ is critical then $ w(C) < \frac{|C|q}{|C|-1}$.
\end{lem}
\begin{proof}
 Let $C$ be a coalition with $w_i \in C$, since $w_i$ is a critical player for the coalition we have the inequality $w(C) -w_i < q$ or
 \[w(C) < q+ w_i\]
 If we sum both sides of the inequality we get \[ w(C) \vert C \vert < q \vert C \vert + w(C)\]
 Or \[ w(C) ( \vert C \vert -1) < q \vert C \vert \]
 \[w(C) < \frac{q \vert C \vert }{\vert C \vert -1}\]

\end{proof}

Let us define given an association matrix $A$, and a player $p_i$ with weight $w_i$
we define $\sum_{j}a_{ij} w_j =d_i+ a_{ii}w_i= d_i+w_i$; this number in a sense represents the overall relation of the player $p_i$ with the other players in the set of players. When the numbers $d_i$ are all
positive for all $i$ we call such game an overall positive association game. Note that for an association voting game if the association matrix is positive then the game is
also overall positive relation game. In the following proposition we will relate the Banzhaf index of an overall positive voting game with the Banzhaf index of the related game
without association. If $\beta_{\phi,i}$ denote the Banzhaf index of the player
$p_i$ with the association and let $\beta_i$ be the index for the same player without the association. Let us also define $w_i + q = q_i$ and let the total weight of a coalition $C$
be denoted by $w(C)$.

\begin{prop}\label{asso}
With the notation as in the above we have the following statements relating the actual Banzhaf index with the association based Banzhaf index.
 \begin{enumerate}
  \item $\beta_{\phi,i} \leq \beta_i$ for all $i$ such that $d_i \geq 0$
  \item $\beta_{\phi,i} \geq \beta_i$ for all $i$ such that $d_i \leq 0$
  \item $\beta_{\phi,i} - \beta_i$ is precisely the number of coalitions $C$ such that $ d_i+ q_i < w(C) < q_i$ minus the number of coalitions such that
  $q_i < w(C) < d_i+ q_i$

 \end{enumerate}

\end{prop}
\begin{proof}
 Note that the first and the second statements follow from the third so we will prove the third statement. Let us define $v(C)$ for a coalition to be 1 if $w(C) > q$ and 0 otherwise
 and let us define ${v'}_i(C)$ to be $1$ if $w(C)- d_i -w_i > q$. With these functions defined it is easy to observe that
 $\beta_{i}= \frac{1}{2^{n-1}} \displaystyle{\sum v(C)- v(C \setminus i)}$ and $\beta_{\Phi, i}= \frac{1}{2^{n-1}} \displaystyle{\sum v(C)- v'(C)}$ where the sums are running over all
 possible coalitions.
 \[\beta_{\Phi,i} - \beta_i = \frac{1}{2^{n-1}} \Big( \displaystyle{\sum v(C)- {v'}_i(C)}- \displaystyle{\sum v(C)- v(C \setminus i)}\Big)\]
 Or
  \[\beta_{\Phi,i} - \beta_i = \frac{1}{2^{n-1}}\Big( \displaystyle{\sum v(C \setminus i) - {v'}_i(C)} \Big )\]

 Let us analyse $v(C \setminus i) - {v'}_i(C)$ depending on the various possibilities in the following way. Note that if a coalition satisfies
 $w(C) > \max\{ q_i+ d_i, q_i\}$, or $w(C) < \min\{ q_i+ d_i, q_i\}$ then both $v(C \setminus i)$ and ${v'}_i(C)$ are equal to 1, rest 0 making the difference 0,
 similarly if $C$ satisfies $q_i +d_i < w(C) < q_i$ then the difference is 1, and if $q_i < w(C) < q_i+d_i$ then the difference is -1. This immediately gives the result.
\end{proof}

In the following proposition we will find a few general bounds for the power indices for a general voting game. Let us define for a set of players with weights $W=\{w_1,w_2, w_3, \ldots
w_n\}$ and a quota $q$ two numbers $m,M$ in the folllowing way. Let $m$ be the smallest natural number such that $m \max{ W} < q$ and $ M \min{W} - max{W} > q$ where $\min{W}, \max{W}$
denotes the minimum and the maximum weights among $W$ respectively.

\begin{prop}\label{prop1}
With the above notations we have the following inequalities for the indices.
\begin{enumerate}
 \item $\beta_i \leq \displaystyle{\frac{ \displaystyle{\sum_{i=m+1}^{\min\{M,n\}}  \binom{n}{i}- 2^{n-1}} } {2^{n}}}$
 \item $\beta_i \leq \displaystyle\frac{\displaystyle{\sum_{i=m+1}^{\min\{M,n\}}  i \binom{n}{i} }}{n2^{n}} -\frac{1}{2}$
\end{enumerate}
\begin{proof}
 Note that due to the definition of the numbers $m,M$ any subset of the set of players of size less than $m$ cannot be critical for any player since the
 total weight does not exceed $q$ and similarly any coalition of size more than $M$ cannot be critical since removing even the maximum weight player does not
 reduce the weight of the coalition below $q$ so the coalitions that are relevant for the indices are of size between $m+1$ and $\min\{M,n\}$. Now out of these
 coalitions there are $2^{n-1}$ that do not contain the player indexed $i$ thus we get the first inequality. For the next inequality notice if all the players are critical in all
 possible coalitions we have all indices equal and maximum possible. In this hypothetical situation we have the total number of pairs $(p_i,C)$ where $p_i \in C$ critical is
 the sum $\displaystyle{\sum_{i=m+1}^{\min\{M,n\}} i \binom{n}{i} - n 2^{n-1}}$ so now dividing the number by $n$ and normalizing we get the result.
\end{proof}
\end{prop}

\begin{rem}\label{rem2}
 Note that the above proposition \ref{prop1} is a generalization of \cite[Dubey '79]{dubey} as the binomial coefficients $\binom{n}{i}$ in the part one of the inequality above
 is replaced by the largest binomial coefficient $\binom{n}{m}$ for $m= \floor{\frac{n}{2}}$ and then we sum the index for all $i$.
\end{rem}
\begin{rem}\label{rem3}
 Note that the lemma \ref{bound3} shows that an assumption such as all players are critical for any coalition is not at all realistic as it puts severe condition on the
 size of the coalition. This also further shows there are much more scope to improve such bounds by scrutinizing the critical coalitions in detail.
\end{rem}

\begin{rem}\label{rem3}
 For a large set of players with a small variation in the weights the above bounds in \ref{bound3} give better result in the simulation done in the next section.
\end{rem}

In our calculation we have observed that for the relative Banzhaf index $\beta_i$ the following relation is valid for any voting game. We have not been able to prove the
claim yet but we state it as a Conjecture as it is an interesting and a very narrow bound.

\begin{conj}\label{lem11}
 Let $w_1,w_2,w_3, \ldots w_m$ be a set of positive weights, let $w$ be the maximum of these weights, further let $N= \sum_{i=1}^{m} w_i$ then the relative
 Banzhaf indices $\beta_i \leq \frac{2w}{N}$.

\end{conj}

\section{Error Bounds and Convergence Analysis}
It is easy to show that the estimator for the Banzhaf index used in \ref{ssec:Calc} is an unbiased estimator
\begin{lem}\label{lem1}
Let $C_1,C_2,\cdots,C_n$ be set of randomly sampled coalitions containing player $p_i$ and $X_1,X_2,\cdots,X_n$ be the random variables as defined in previous section. Then the estimator $\hat{\beta}_{\Phi,i}$ is an unbiased estimator for ${\beta}_{\Phi,i}.$
\end{lem}
\begin{proof}
  The random variable $X_k$ is a Bernoulli random variable with probability of success ${\beta}_{\Phi,i}$, so the random variable X = $\sum_{k=1}^{n}X_k$ is Binomial with mean $n\beta_{\Phi,i}$. The maximum likelihood estimator for the parameter ${\beta}_{\Phi,i}$  is $\hat{\beta}_{\Phi,i}$ = $\frac{X}{n}$. Since $E\left(\frac{X}{n}\right)=\beta_{\Phi,i}$ it implies that the estimator $\hat{\beta}_{\Phi,i}$ is an unbiased maximum likelihood estimator.
\end{proof}
\subsection{Confidence intervals}
Before trying to calculate, the confidence interval by using different approximation we first state few results which will help to find the confidence bounds.
\begin{defn}
A function $f:X^n \rightarrow\mathbb{R}$  is said to have self-bounding property if there exist a function $f_i:X^{n-1} \rightarrow\mathbb{R}$ such that for all $x\in X^n$ and $1\leq i\leq n$,
\begin{itemize}
  \item[(a)] $0\leq f(x)-f_i(x^{(i)})\leq 1.$
  \item[(b)] $\sum_{i=1}^{n}\left(f(x)-f_i(x^{(i)})\right)\leq f(x),$ where $x^{(i)}$ is obtained by dropping the $i^{th}$ component of $x$~ i.e., $x^{(i)}=(x_1,x_2\cdots,x_{i-1},x_{i+1},\cdots,x_n)\in X^{n-1}.$
\end{itemize}
\end{defn}
Several generalization and weaker version of self- bounding (one can refer \cite{bouch,bouch1}) are defined by the following definitions.
\begin{defn}
A function $f:X^n \rightarrow\mathbb{R}$ is called $(a,b)-$ self-bounding if there exist positive numbers $a$ and $b$ such that for all $x\in X^n$ and $1\leq i\leq n$,
\begin{itemize}
  \item[(a)] $0\leq f(x)-f_i(x^{(i)})\leq 1.$
  \item[(b)] $\sum_{i=1}^{n}\left(f(x)-f_i(x^{(i)})\right)\leq af(x)+b.$
  \end{itemize}
\end{defn}
\begin{defn}
A function $f:X^n \rightarrow\mathbb{R}$ is called weakly
$(a,b)-$ self-bounding if there exist positive numbers $a$ and $b$ such that for all $x\in X^n$ and $1\leq i\leq n$,
\begin{equation*}
 \sum_{i=1}^{n}\left(f(x)-f_i(x^{(i)})\right)^2\leq af(x)+b.
\end{equation*}
\end{defn}
\begin{lem}\label{selfbounding}
The function $f: \chi^{n} \rightarrow \mathbb{R}$ given by $f(x_1,x_2, \ldots , x_n)= \frac{1}{n} \sum_{i=1}^{n} x_i$ is selfbounding and $(1/n,0)$ weakly selfbounding.
\end{lem}

\begin{proof}
Let us take $f_i(x^i)= \frac{1}{n}\sum_{j=1 , j \neq i}^n x_j$ then $f- f_i(x^i)= \frac{x_i}{n}$ thus $0 \leq (f- f_i (x^i) \leq 1$ for all $i$. And further $\sum_{i=1}^{n} (f-f_i(x^i)) = \frac{1}{n} x_i=f$ which proves that the function is selfbounding.
And $\sum_{i=1}^{n} (f-f_i(x^i))^2 = \frac{x_i^2}{n^2} = \frac{1}{n^2} \sum_{i=1}^{n} x_i^2 \leq \frac{1}{n^2} \sum_{i=1}^{n} x_i$. Since  $\sum_{i=1}^{n} (f-f_i(x^i))^2 \leq  \frac{1}{n^2} \sum_{i=1}^{n} x_i$ we get the weakly selfbounding
for $a=\frac{1}{n}$ and $b=0$.
\end{proof}

\begin{rem}
 Note that any positive constant multiple of a selfbounding function ($(a,0)$ weakly selfbounding) is also selfbounding (weakly $(a,0)$ selfbounding ).
\end{rem}
To draw best estimates and bounds for any random variable when the variance
information is unknown, Hoeffding’s inequality\cite{hoe} is most  preferable to use.
\begin{thm}[Hoeffding’s inequality]
  Let $X_1,X_2,\cdots,X_n$ be independent and identically distributed random variables and $Z=\frac{1}{n}\sum_{i=1}^{n}X_i.$ If for each $1\leq i\leq n,$ $X_i$ is bounded by the interval $[a,b],$ then the following holds.
  \begin{equation*}
    \mbox{Prob}\left\{|Z-E(Z)|<\epsilon\right\}\geq 1-2\exp\left(-\frac{2n\epsilon^2}{(b-a)^2}\right).
  \end{equation*}
\end{thm}
Now, we state few basic results in the context of self bounding and weak self bounding which we will use for finding one-sided confidence intervals for the indices.
\begin{lem}\label{self}
Let $X_1,X_2,\cdots,X_n$ be independent and identically distributed random variables. If $Z=f(X_1,X_2,\cdots,X_n)$
and $f$ is weakly $(a,0)$ self bounding then for all $t>0,$
$$
    \mbox{Prob}\left\{(Z-E(Z))\geq t\right\}\leq\exp\left(-\frac{t^2}{2aE(Z)+at}\right).
$$
 Further, if $f$ has $(a,0)$ self bounding then \\
   $$ \mbox{Prob}\left\{(Z-E(Z))\leq -t\right\}\leq\exp\left(-\frac{t^2}{2\max(a,1)E(Z)}\right).
$$
\end{lem}
In particular, if $f:X^n \rightarrow\mathbb{R}^n$ is defined as $f(x)=Z=\sum_{k=1}^{n}X_k.$  Since $E(Z)=n\beta_i$ and $f$ is $(1,0)$ self bounding. Then by using Lemma \ref{self} we get  better estimates as the followings
  \begin{equation}\label{self1}
    \mbox{Prob}\left\{(Z-E(Z))\geq n\epsilon\right\}\leq\exp\left(-\frac{n\epsilon^2}{2\beta_i+\epsilon}\right) ~~\mbox{ and}
  \end{equation}
    \begin{equation}\label{wself}
    \mbox{Prob}\left\{(Z-E(Z))\leq - n\epsilon\right\}\leq\exp\left(-\frac{n\epsilon^2}{2}\right).
  \end{equation}
  Based on the above results, we arrive a statement to give a precise bound for the estimator and writing as the following lemma
  \begin{lem}[Confidence intervals]\label{conf}
  For a given $n$ randomly samples and accuracy $\epsilon>0$ with confidence level atleast $1-\delta$, the following confidence intervals holds:
\begin{eqnarray*}
   &\mbox{a. }&\left[\hat{\beta}_{\Phi,i}-\sqrt{\frac{1}{2n}\ln\left(\frac{2}{\delta}\right)},~\hat{\beta}_{\Phi,i}+
        \sqrt{\frac{1}{2n}\ln\left(\frac{2}{\delta}\right)}\right] \\
   &\mbox{b. }& \left[\hat{\beta}_{\Phi,i}-t_{\delta/2}\frac{S}{\sqrt{n}},~\hat{\beta}_{\Phi,i}+
        t_{\delta/2}\frac{S}{\sqrt{n}}\right]
\end{eqnarray*}
\end{lem}
\begin{proof}
We are looking an $\epsilon$ for which $P(|\hat{\beta}_{\Phi,i}-\beta_{\Phi,i}|<\epsilon)$ is atleast $1-\delta.$ Since $X_i$'s are bounded by $[0,1]$, so by using Hoeffding’s inequality, the $\epsilon$ should satisfy
\begin{eqnarray*}
 && 1-2e^{ -{2n\epsilon^2}} \geq 1-\delta,~~~ \Rightarrow \delta \geq 2e^{-{2n\epsilon^2}} \\
 &\Rightarrow&  \ln\left(\frac{2}{\delta}\right) \leq 2n\epsilon^2,~~~ \Rightarrow \epsilon \geq \sqrt{\frac{\ln\left(\frac{2}{\delta}\right)}{2n}}
\end{eqnarray*}
This proves first part of the lemma.
Since $X_1,X_2,\cdots,X_n$ are independent and Bernoulli trials with mean $\displaystyle\mu=\beta_{\Phi,i}$ and variance $\displaystyle\sigma^2= \beta_{\Phi,i}(1-\beta_{\Phi,i}).$ So for large $n,$ and using central limit theorem, the random variable $\displaystyle\frac{\overline{X}-\beta_{\Phi,i}}{\sigma/\sqrt{n}}$ is approximately standard normal. If $\sigma$ is unknown and $n$ is large, then $\displaystyle\frac{\overline{X}-\beta_{\Phi,i}}{S/\sqrt{n}}$ is approximately $t-$ distribution with degrees of freedom $(n-1),$ where $S^2$ is the sample variance and defined by $S^2=\displaystyle\frac{1}{n-1}\sum_{i=1}^{n}(X_i-\overline{X})^2.$ Let us denote $P(t>t_{\alpha})=\alpha.$ Using this notation, $P(|\overline{X}-\beta_{\Phi,i}|\leq \epsilon)\geq 1-\delta$ implies $\displaystyle\frac{\epsilon\sqrt{n}}{S}\geq t_{\delta/2}.$ Therefor $\epsilon \geq \displaystyle\frac{St_{\delta/2}}{\sqrt{n}}.$ Which completes the second part of the theorem.
\end{proof}
From the above Lemma, it can be observed that the number samples required to get the accuracy $\epsilon$ is at least $\displaystyle\left[\frac{\ln\left(2/\delta\right)}{2\epsilon^2}\right]$ or $\displaystyle\frac{S^2t_{\delta/2}^2}{\epsilon^2}$ in case of part $(b).$
\begin{thm}\label{bconf}
Let $X_1,X_2,\cdots,X_n$ be $n$ random samples. Then for given accuracy $\epsilon>0,$ the confidence interval for $\beta_{\Phi,i}$ (with confidence level $1-\delta$) is
$$
  \left(\hat{\beta}_{\Phi,i}-\frac{1}{\sqrt{n}}\sqrt{B\ln\left(\frac{2}{\delta}\right)},~~~\hat{\beta}_{\Phi,i}+
        \frac{1}{\sqrt{n}}\sqrt{B\ln\left(\frac{2}{\delta}\right)}\right).
$$
\end{thm}
\begin{proof}
  By using equation \ref{self1},we obtain
\begin{equation*}
  P\left\{\hat{\beta}_{\Phi,i}-\beta_{\Phi,i}>\epsilon\right\}\leq e^{\left(-\frac{n\epsilon^2}{2\beta_i+\epsilon}\right)}.
  \end{equation*}
 Using proposition \ref{prop1}, and denoting the bound for $2\beta_i+\epsilon$ as $B=B_{w,q,\epsilon},$ we have the following bound:
 \begin{equation*}
     P\left\{\hat{\beta}_{\Phi,i}-\beta_{\Phi,i}>\epsilon\right\}\leq e^{\left(-\frac{n\epsilon^2}{B}\right)}.
  \end{equation*}
Since the random variable $\hat{\beta}_{\Phi,i}-\beta_{\Phi,i}$ is approximately normal distribution with mean  zero.
\begin{eqnarray*}
 \Rightarrow e^{\left(-\frac{n\epsilon^2}{B}\right)} &\leq& \left(\frac{\delta}{2}\right) \\
  \epsilon &\geq  & \frac{1}{\sqrt{n}}\sqrt{B\ln\left(\frac{2}{\delta}\right)}.
\end{eqnarray*}
 This completes the proof.
\end{proof}
\begin{rem}
 The number of samples required within accuracy $\epsilon$ is atleast $\frac{B\ln\left(\frac{2}{\delta}\right)}{\epsilon^2}.$
\end{rem}
\section{Experimental Results}
In this section, we will discuss the comparison of $3-$quota banzhaf indices without association (WTA) and with the association matrix(WA). In the comparison analysis, the European Union voting system \cite{nice} is considered along with the following quota.
\begin{itemize}
\item[(a)] Weight quota ($74\%$ of voting weights).
\item[(b)] Population quota ( $62\%$ of population).
\item[(c)] Majority of the number of countries ($50\% + 1$).
\end{itemize}

For our analysis we consider the countries with weights more than equal to 7,i.e, $w \geq 7$. The weight-data and population data can be found in \footnote{http://www.consilium.europa.eu/en/council-eu/voting-system/qualified-majority/} \footnote{https://en.wikipedia.org/wiki/List\textunderscore of\textunderscore European\textunderscore Union\textunderscore member\textunderscore states\textunderscore by\textunderscore population}.
 The association matrix  $\Phi$ is taken based on the immigration  between the respective EU countries. We define the association matrix in the following manner.\\
 Let us denote $M_{i,j}$ be the number of people migrating from country $i$ to country $j$  and $M_{j,i}$ be the number of people migrating from country $j$ to country $i.$ Define $M=\displaystyle\max_{i \neq j}|M_{i,j}-M{j,i}|.$ Now the association matrix $\Phi_{i,j}$ defined as
$$\Phi_{i,j}=\left\{\begin{array}{ccc}
  1& ~~\mbox{ If $i=j.$}\\
  \frac{M_{j,i}-M_{i,j}}{M} &~~\mbox{ if $1\leq j<i\leq 18.$}\\
-\Phi_{j,i}&~~\mbox{if $1\leq i<j\leq 18.$}
\end{array}\right.$$
Using the above immigration association matrix, improved Banzhaf indices are computed and given in Table \ref{tab:Table1}.
\begin{table}[H]
	\centering
\caption{Without Association v/s with association(based on immigration)}
	\label{tab:Table1}
	\begin{tabular}{|c| c| c| c|c|c|}
		\hline
		Country & Weight & Popln & WTA & WA-immgr\\
		\hline
		Austria(AUT)        & 10 & 8.58 & 0.03549 & 0.05554 \\
		Belgium(BEL)        & 12 & 11.25 & 0.04403 & 0.06485 \\
		Czech Republic(CZE) & 12 & 10.53 & 0.04403 & 0.04681 \\
		Germany(DEU)        & 29 & 82.30 & 0.09560 & 0.12898 \\
		Denmark(DNK)        & 7 & 5.66   & 0.02629 & 0.03317 \\
		Spain(ESP)          & 27 & 46.46 & 0.08853 & 0.10302 \\
		Finland(FIN)        & 7 & 5.47   & 0.02629 & 0.03097 \\
		France(FRA)         & 29 & 66.99 & 0.09560 & 0.11263 \\
		Britain(GBR)        & 29 & 65.11 & 0.09560 & 0.12486 \\
		Greece(GRC)         & 12 & 10.81 & 0.04403 & 0.02750 \\
		Hungary(HUN)        & 12 & 9.85 & 0.04403 & 0.00396 \\
		Ireland(IRL)        & 7 & 4.63  & 0.02629 & 0.03236  \\
		Italy(ITA)          & 29 & 60.79 & 0.09560 & 0.06765 \\
		Netherlands(NLD)    & 13 & 17.10 & 0.04418 &0.06505  \\
		Poland(POL)         & 27 & 38.56 & 0.08853 & 0.0\\
		Portugal(PRT)       & 12 & 10.37 & 0.04403 & 0.03527\\
		Slovakia(SVK)       & 7 & 5.42   & 0.02629 & 0.01762 \\
		Sweden(SWE)         & 10 & 10.01 & 0.03549 &0.04968 \\
		All 18 Countries    & 291& 469.93& 1.0     &1.0\\		
		\hline
	\end{tabular}
	\end{table}
Now, we consider a random matrix with entries from $-1$ to $1$ as association matrix between EU countries. The Banzhaf indices with random association (WA-Random) are obtained by averaging the value of the banzhaf indices over 100 runs, where we take a new random association matrix in each run and the Banzhaf indices without association (WTA) are obtained in a single run. The computed indices are provided in Table \ref{tab:Table2}.
\begin{table}[H]
	\centering
\caption{Without Association v/s with random association}
	\label{tab:Table2}
	\begin{tabular}{|c| c| c| c|c|c|}
		\hline
		Country & Weight & Popln & WTA & WA-Random\\
		\hline
		Austria(AUT)        & 10 & 8.58 & 0.03549 & 0.04364\\
		Belgium(BEL)        & 12 & 11.25 & 0.04403 & 0.04509 \\
		Czech Republic(CZE) & 12 & 10.53 & 0.04403 & 0.04717 \\
		Germany(DEU)        & 29 & 82.30 & 0.09560 & 0.08706 \\
		Denmark(DNK)        & 7 & 5.66   & 0.02629 & 0.04134 \\
		Spain(ESP)          & 27 & 46.46 & 0.08853 &  0.07813\\
		Finland(FIN)        & 7 & 5.47   & 0.02629 & 0.04330 \\
		France(FRA)         & 29 & 66.99 & 0.09560 & 0.08860 \\
		Britain(GBR)        & 29 & 65.11 & 0.09560 & 0.08013\\
		Greece(GRC)         & 12 & 10.81 & 0.04403 & 0.04831 \\
		Hungary(HUN)        & 12 & 9.85 & 0.04403 & 0.04863 \\
		Ireland(IRL)        & 7 & 4.63  & 0.02629 & 0.03904 \\
		Italy(ITA)          & 29 & 60.79 & 0.09560 & 0.07439 \\
		Netherlands(NLD)    & 13 & 17.10 & 0.04418 & 0.04785\\
		Poland(POL)         & 27 & 38.56 & 0.08853 & 0.06762\\
		Portugal(PRT)       & 12 & 10.37 & 0.04403 & 0.04858\\
		Slovakia(SVK)       & 7 & 5.42   & 0.02629 & 0.03724 \\
		Sweden(SWE)         & 10 & 10.01 & 0.03549 & 0.03380\\
		All 18 Countries    & 291& 469.93& 1.0     & 1.0\\		
		\hline
	\end{tabular}
	\end{table}
The approximated Banzhaf indices without association (WTA-Approx) and with random association matrix (WA-Random-Approx) are computed by using  \textbf{Algorithm} \ref{algo:2} which discussed in Section 2.
Note that while computing the indices, we use a single random matrix and obtain the exact value of WA-Random. and for computing the approximated indices, we run the algorithm $20$ times  and then averaged. Similar computation followed for WTA-Approx. Also we consider $\epsilon = 0.01$ and $\delta=0.01$ for each of the approximation algorithm runs.
\begin{table}[H]
	\centering
\caption{Exact v/s Approximation}
	\label{tab:Table3}
	\begin{tabular}{|c| c| c| c|c|c|}
		\hline
		Country & WTA & WTA-Approx & WA-Random & WA-Random-Approx\\
		\hline
		AUT  & 0.03549 & 0.03540 & 0.06774 & 0.06757 \\
		BEL  & 0.04403 & 0.04428 & 0.07138 & 0.07150\\
		CZE  & 0.04403 & 0.04384 & 0.07138 & 0.07138\\
		DEU  & 0.09560 & 0.09557 & 0.09339 & 0.09353\\
		DNK  & 0.02629 & 0.02629 & 0.06305 & 0.06277\\
		ESP  & 0.08853 & 0.08844 &  0.0    & 0.0\\
		FIN  & 0.02629 & 0.02621 &  0.03484    & 0.03474\\
		FRA  & 0.09560 & 0.09579 &  0.02633 & 0.02653\\
		GBR  & 0.09560 & 0.09588 &  0.06871 & 0.06916\\
		GRC  & 0.04403 & 0.04392 &  0.00084 &  0.00083\\
		HUN  & 0.04403 & 0.04436 &  0.06967 & 0.06997\\
		IRL  & 0.02629 & 0.02624 &  0.06383 & 0.06371\\
		ITA  & 0.09560 & 0.09532 &  0.08334 & 0.08319\\
		NLD  & 0.04418 & 0.04429 &  0.01747 & 0.01740\\
		POL  & 0.08853 & 0.08871 &  0.07458 & 0.07480\\
		PRT  & 0.04403 & 0.04395 &  0.06468 & 0.06457\\
		SVK  & 0.02629 & 0.02617 &  0.06245 & 0.06228\\
		SWE  & 0.03549 & 0.03526 &  0.06623 & 0.06597\\	
		\hline
	\end{tabular}
	\end{table}
The result displayed in Table \ref{tab:Table3} justifies that the Banzhaf indices $\hat{\beta_{\Phi,i}}$ are obtained satisfies $P(|\hat{\beta_{\Phi,i}}-\beta_{\Phi,i}| > \epsilon) \leq \delta $.
\section{Conclusion}
We have an improved and generalized power index that captures the internal dynamic of players in real life situations that typically arise as examples of voting games.
The few main lines of future research directions that can be picked up from here are that of showing theoretical results regarding the index. Better approximations of the
index as in \cite{dubey} could be also interesting to investigate. In the direction of application one can look into ways of estimating the association matrix data
in real life, and the subsequent theoretical analysis of the effect of the association matrix on the index. In the third line we would like mention that with a restricted
domain of weights such as weights satisfying some conditions as in \cite{} is a very interesting future direction. A complete development of axiomatics of the improved
index as in \cite{burgin01, dubey} is also a theoretical problem.

\nocite{*}


\begin{thebibliography}{99}

\bibitem{algaba} E. Algaba, J. M. Bilbao, J. F. Garcıa, and J. L'opez. Computing power indices in
weighted multiple majority games. Mathematical Social Sciences, 46(1):63–80,2003.

\bibitem{inverse} N. Alon and P. H. Edelman. The inverse banzhaf problem. Social Choice and
Welfare, 34(3):371–377, 2010.

\bibitem{Monte2} Y. Bachrach, E. Markakis, A. D. Procaccia, J. S. Rosenschein, and A. Saberi.
Approximating power indices. In Proceedings of the 7th international joint conference
on Autonomous agents and multiagent systems-Volume 2, pages 943–950.
International Foundation for Autonomous Agents and Multiagent Systems, 2008.
\bibitem{Monte1} Y. Bachrach, E. Markakis, E. Resnick, A. D. Procaccia, J. S. Rosenschein, and
A. Saberi. Approximating power indices: theoretical and empirical analysis. Autonomous
Agents and Multi-Agent Systems, 20(2):105–122, 2010.
\bibitem{network} Y. Bachrach and J. S. Rosenschein. Computing the banzhaf power index in network
flow games. In Proceedings of the 6th international joint conference on
Autonomous agents and multiagent systems, page 254. ACM, 2007.
\bibitem{power} Y. Bachrach and J. S. Rosenschein. Power in threshold network flow games.
Autonomous Agents and Multi-Agent Systems, 18(1):106–132, 2009.
\bibitem{banzhaf68} J. Banzhaf. One man,? votes: Mathematical analysis of political consequences
and judicial choices. George Washington Law Rev, 36:808–823, 1968.
\bibitem{banzhaf} J. F. Banzhaf III. Weighted voting doesn’t work: A mathematical analysis. Rutgers
L. Rev., 19:317, 1964.
\bibitem{Convex} J. Bilbao, A. Jim'enez, and J. L'opez. The banzhaf power index on convex geometries.
Mathematical Social Sciences, 36(2):157–173, 1998.
\bibitem{bouch1} S. Boucheron, G. Lugosi, and P. Massart. Concentration inequalities: A
nonasymptotic theory of independence. Oxford university press, 2013.
\bibitem{bouch} S. Boucheron, G. Lugosi, P. Massart, et al. On concentration of self-bounding
functions. Electronic Journal of Probability, 14:1884–1899, 2009.
\bibitem{burgin01} M. Burgin and L. Shapley. Enhanced banzhaf power index and its mathematical
properties. WP-797, Department of Mathematics, UCLA, 2001.
\bibitem{genetic} E. K. Burke, M. Hyde, G. Kendall, and J. Woodward. A genetic programming
hyper-heuristic approach for evolving 2-d strip packing heuristics. IEEE Transactions
on Evolutionary Computation, 14(6):942–958, 2010.
\bibitem{compound} P. Dubey, E. Einy, and O. Haimanko. Compound voting and the banzhaf index.
Games and Economic Behavior, 51(1):20–30, 2005.
\bibitem{dubey} P. Dubey and L. S. Shapley. Mathematical properties of the banzhaf power index.
Mathematics of Operations Research, 4(2):99–131, 1979.

\bibitem{Heu} S. Fatima, M.Wooldridge, and N. R. Jennings. A heuristic approximation method
for the banzhaf index for voting games. Multiagent and Grid Systems, 8(3):257–
274, 2012.
\bibitem{axiom2} V. Feltkamp. Alternative axiomatic characterizations of the shapley and banzhaf
values. International Journal of Game Theory, 24(2):179–186, 1995.
\bibitem{EU_tempt} G. Garrett and G. Tsebelis. Even more reasons to resist the temptation of power
indices in the eu. Journal of Theoretical Politics, 13(1):99–105, 2001.
14
\bibitem{axiom3} M. Grabisch and M. Roubens. An axiomatic approach to the concept of interaction
among players in cooperative games. International Journal of Game Theory,
28(4):547–565, 1999.
\bibitem{hoe} W. Hoeffding. Probability inequalities for sums of bounded random variables.
Journal of the American statistical association, 58(301):13–30, 1963.
\bibitem{Bounds} R. Holzman, E. Lehrer, and N. Linial. Some bounds for the banzhaf index and
other semivalues. Mathematics of Operations Research, 13(2):358–363, 1988.
\bibitem{kilgour1983formal} D. M. Kilgour. A formal analysis of the amending formula of canada’s constitution
act, 1982. Canadian Journal of Political Science/Revue canadienne de
science politique, 16(4):771–777, 1983.
\bibitem{laruelle2001shapley} A. Laruelle and F. Valenciano. Shapley-shubik and banzhaf indices revisited.
Mathematics of operations research, 26(1):89–104, 2001.
\bibitem{Eu_fair} A. Laruelle and M. Widgr'en. Is the allocation of voting power among eu states
fair? Public Choice, 94(3):317–339, 1998.
\bibitem{Large} D. Leech. Computing power indices for large voting games. Management Science,
49(6):831–837, 2003.
\bibitem{survey} T. Matsui and Y. Matsui. A survey of algorithms for calculating power indices of
weighted majority games. Journal of the Operations Research Society of Japan,
43(1):71–86, 2000.
\bibitem{Approximations} S. Merrill. Approximations to the banzhaf index of voting power. The American
Mathematical Monthly, 89(2):108–110, 1982.
\bibitem{naive} C. H. Nevison, B. Zicht, and S. Schoepke. A naive approach to the banzhaf index
of power. Systems Research and Behavioral Science, 23(2):130–131, 1978.
\bibitem{extension} G. Owen. Multilinear extensions and the banzhaf value. Naval Research Logistics
(NRL), 22(4):741–750, 1975.
\bibitem{NP} K. Prasad and J. S. Kelly. NP-completeness of some problems concerning voting
games. International Journal of Game Theory, 19(1):1–9, 1990.
\bibitem{S_values} A. E. Roth. The Shapley value: essays in honor of Lloyd S. Shapley. Cambridge
University Press, 1988.
\bibitem{SS_prob} P. Straffin. The shapley—shubik and banzhaf power indices as probabilities. The
Shapley value. Essays in honor of Lloyd S. Shapley, pages 71–81, 1988.
\bibitem{straf} P. D. Straffin. Homogeneity, independence, and power indices. Public Choice,
30(1):107–118, 1977.
\bibitem{prob_model} P. D. Straffin. Probability models for power indices. Game theory and political
science, pages 477–510, 1978.
\bibitem{nice} N. Treaty. http://europa.eu/rapid/press-release memo-03-234 en.htm, 2003.
\bibitem{weber} M. Weber. Two-tier voting: Measuring inequality and specifying the inverse
power problem. Mathematical Social Sciences, 79:40–45, 2016.


\end{thebibliography}
\end{document}